\documentclass{cmslatex}
\usepackage[paperwidth=7in, paperheight=10in, margin=.875in]{geometry}
\usepackage[colorlinks=true, linkcolor=red, citecolor=blue]{hyperref}
\usepackage{amsfonts,amssymb}
\usepackage{amsmath}
\usepackage{graphicx}
\usepackage{cite}
\usepackage{enumerate}
\renewcommand{\theequation}{\arabic{section}.\arabic{equation}}

   \allowdisplaybreaks
\begin{document}
\title{Boundary controllability of the Korteweg-de Vries equation: The Neumann case
}


          \author{Roberto de A. Capistrano--Filho \thanks{Departamento de Matem\'atica,  Universidade Federal de Pernambuco (UFPE), 50740-545, Recife (PE), Brazil, (\href{mailto:roberto.capistranofilho@ufpe.br}{roberto.capistranofilho@ufpe.br}).}
          \and Jandeilson Santos da Silva \thanks{Departamento de Matem\'atica,  Universidade Federal de Pernambuco (UFPE), 50740-545, Recife (PE), Brazil (\href{mailto:jandeilson.santos@ufpe.br}{jandeilson.santos@ufpe.br}}).}

         \pagestyle{myheadings} \markboth{CRITICAL SET PHENOMENON FOR THE KDV EQUATION}{CAPISTRANO-FILHO AND DA SILVA} \maketitle

          \begin{abstract}
This article gives a necessary first step to understanding the critical set phenomenon for the Korteweg-de Vries (KdV) equation posed on the interval $[0,L]$, considering the Neumann boundary conditions with only one control input. We showed that the KdV equation is controllable in the critical case, i.e., when the spatial domain $L$ belongs to the set $\mathcal{R}_c$, where $c\neq-1$ and 
$$
 \mathcal{R}_c:=\left\{\frac{2\pi}{\sqrt{3(c+1)}}\sqrt{m^2+ml+m^2};\ m,l\in \mathbb{N}^*\right\}\cup\left\{\frac{m\pi}{\sqrt{c+1}};\ m\in \mathbb{N}^*\right\},
$$
the KdV equation is exactly controllable in $L^2(0,L)$. The result is achieved using the \textit{return method} together with a fixed point argument.
          \end{abstract}

          \section{Introduction}\label{intro}
          
We had known when we formulated the waves as a free boundary problem of the incompressible, irrotational Euler equation in an appropriate non-dimensional form, there exist two non-dimensional parameters $\delta:= \frac{h}{\lambda}$ and $\varepsilon:= \frac{a}{h}$, where the water depth, the wavelength and the amplitude of the free surface are parameterized as $h, \lambda$ and $a$, respectively. See, for instance, \cite{BCS, BCL, ASL, BLS, Lannes, Saut} and references therein for a rigorous justification. Moreover, another non-dimensional parameter $\mu$ appears, the Bond number, to measure the importance of gravitational forces compared to surface tension forces. Considering the physical condition $\delta \ll 1$, we can characterize the waves, called long waves or shallow water waves.  In particular, considering the relations between $\varepsilon$ and $\delta$, we can have the following regime:
\begin{itemize}
\item Korteweg-de Vries (KdV): $\varepsilon = \delta^2 \ll 1$ and $\mu \neq \frac13$.
Under this regime, Korteweg and de Vries \cite{Korteweg}\footnote{This equation was first introduced by Boussinesq \cite{Boussinesq}, and Korteweg and de Vries rediscovered it twenty years later.} derived the following  well-known equation as a central equation among other dispersive or shallow water wave models called the KdV equation from the equations for capillary-gravity waves: 
\[\pm2 \eta_t + 3\eta\eta_x +\left( \frac13 - \mu\right)\eta_{xxx} = 0.\]
\end{itemize}

Today, it is well known that this equation has an important phenomenon that directly affects the control problem related to it, the so-called \textit{critical length phenomenon}. Let us briefly present the control problem, which makes the phenomenon of critical lengths emerge. The control problem was presented in a pioneering work of Rosier \cite{Rosier} that studied the following system 
\begin{equation}
\left\{
\begin{array}
[c]{lll}%
u_t+u_x+uu_x+u_{xxx}=0, &  & \text{ in } (0,L)\times(0,T),\\
u(0,t)=0,\text{ }u(L,t)=0,\text{ }u_x(L,t)=g(t) ,&  & \text{ in }(0,T),\\
u(x,0)=u_0(x), & & \text{ in }(0,L),
\end{array}
\right.  \label{2}
\end{equation}
where the boundary value function $g(t)$ is considered as a control input. Precisely, the author showed the following control problem for the system \eqref{2}, giving the origin of the critical length phenomenon for the KdV equation.

\vspace{0.2cm}
\noindent\textbf{Question $\mathcal{A}$:} \textit{Given $T>0$ and $u_0,u_T\in L^2(0,L)$, can one find  an appropriate control input $g(t)\in L^2(0,T)$ such that the corresponding solution $u(x,t)$ of the system \eqref{2} satisfies
\begin{equation}\label{control-r}
u(x,0)=u_0(x) \quad \text{and} \quad u(x,T)=u_T(x)?
\end{equation}}

Rosier answered the previous question in \cite{Rosier}. He proved that considering $L\notin\mathcal{N}$, where 
$$
\mathcal{N}:=\left\{  \frac{2\pi}{\sqrt{3}}\sqrt{k^{2}+kl+l^{2}}
\,:k,\,l\,\in\mathbb{N}^{\ast}\right\},
$$
the associated linear system \eqref{2} 
\begin{equation}
\left\{
\begin{array}
[c]{lll}%
u_t+u_x+u_{xxx}=0, &  & \text{ in } (0,L)\times(0,T),\\
u(0,t)=0,\text{ }u(L,t)=0,\text{ }u_x(L,t)=g(t), &  & \text{ in }(0,T),\\
u(x,0)=u_0(x) & & \text{ in }(0,L),
\end{array}
\right.  \label{2a}
\end{equation}
is controllable; roughly speaking, if $L\in\mathcal{N}$ system \eqref{2a} is not controllable, that is,  there exists a finite-dimensional subspace of $L^2(0,L)$, denoted by $\mathcal{M}=\mathcal{M}(L)$, which is unreachable from $0$ for the linear system.  More precisely, for every nonzero state $\psi\in\mathcal{M}$, $g\in L^2(0,T)$  and $u\in C([0,T];L^2(0,L))\cap L^2(0,T;H^1(0,L))$ satisfying \eqref{2a} and $u(\cdot,0)=0$, one has $u(\cdot,T)\neq\psi$. 

\vspace{0.2cm}
\begin{definition}
A spatial domain $(0,L)$ is called \textit{critical}  for the system \eqref{2a} if its domain length $L$ belongs to $\mathcal{N}$.
\end{definition}
\vspace{0.2cm}

Following the work of Rosier \cite{Rosier}, the boundary control system of the KdV equation posed on the finite interval $(0,L)$ with various control inputs has been intensively studied  (cf.  \cite{CaCaZh,cerpa1,cerpa2,CoCre,crepeau,GG,GG1,Gui}  and see \cite{cerpatut, RZsurvey} for more complete reviews). Thus, this work gives a necessary step to understanding this phenomenon for the KdV equation with Neumann boundary conditions, completing, in some sense, the results shown in \cite{CaCaZh}. 

\subsection{Problem set} 

In this article, we study a class of distributed parameter control systems described by the Korteweg–de Vries (KdV) equation posed on a bounded domain $(0, L)$ with the Neumann boundary conditions
\begin{equation}\label{NKdV-2}
\left\{
\begin{array}
[c]{lll}%
	u_t+u_x+u_{xxx}+uu_x=0,&&\text{ in }(0,L)\times (0,T),\\
		u_{xx}(0,t)=u_{xx}(L,t)=0,&&\text{ in }(0,T),\\
		u_x(L,t)=h(t),&&\text{ in }(0,T),\\
		u(x,0)=u_0(x),&&\text{ in }(0,L),
	\end{array}
	\right.
\end{equation}
where $h(t)$ will be considered as a control input. Recently, the first author dealt with the control problem related to the system \eqref{NKdV-2}. Precisely, it was proved that their solutions are exactly controllable in a neighborhood of $c$ if the length $L$ of the spatial domain $(0, L)$ does not belong to the set
$$
\mathcal{R}_c:=\left\{\frac{2\pi}{\sqrt{3(c+1)}}\sqrt{m^2+ml+m^2};\ m,l\in \mathbb{N}^*\right\}\cup\left\{\frac{m\pi}{\sqrt{c+1}};\ m\in \mathbb{N}^*\right\},
$$
that is, the relation \eqref{control-r} holds for the solution of the system \eqref{NKdV-2}. The result can be read as follows.

\vspace{0.2cm}
\begin{theorem}[Caicedo, Capistrano--Filho, Zhang \cite{CaCaZh}]\label{control}
Let $T>0$, $c\ne -1$ and $L \notin \mathcal{R}_c$. There exists $\delta>0$ such that for any $u_0,u_T\in L^2(0,L)$ with
$$
\|u_0-c\|_{L^2(0,L)}<\delta \text{ \ and \ }
\|u_T-c\|_{L^2(0,L)}<\delta
$$
one can find $h \in L^2(0,T)$ such that the system \eqref{NKdV-2} admits a unique solution
$$
u \in \mathcal{Z}_T=C\left([0,T];L^2(0,L)\right)\cap L^2\left(0,T;H^1(0,L)\right)
$$
satisfying \eqref{control-r}.
\end{theorem}

\vspace{0.2cm}
As in \cite{Rosier}, the first step is to obtain a control result for the linear system, namely, 
\begin{equation}\label{lpc}
\left\{
\begin{array}
[c]{lll}%
u_t+(1+c)u_x+u_{xxx}=0,&\text{ \ in \ }(0,L)\times (0,T),\\
u_{xx}(0,t)=u_{xx}(L,t)=0,&\text{ \ in \ }(0,T),\\
u_x(L,t)=h(t),&\text{ \ in \ }(0,T),\\
u(x,0)=u_0(x),&\text{ \ in \ }(0,L).
\end{array}
\right.
\end{equation}
Precisely, the authors in \cite{CaCaZh} proved the following result. 

\vspace{0.2cm}
\begin{theorem}[Caicedo, Capistrano--Filho, Zhang \cite{CaCaZh}]\label{control-c} For $c\ne -1$, the linear system \eqref{lpc} is exactly controllable in the space $L^2(0,L)$ if and only if $L\notin \mathcal{R}_c$. Otherwise, that is, if $c=-1$, the system \eqref{lpc} is not exactly controllable in the space $L^2(0,L)$ for any $L>0$.
\end{theorem}

\vspace{0.2cm}

With this in hand, they extend the result for the nonlinear system \eqref{NKdV-2} using a fixed point argument, achieving Theorem \ref{control} whenever $L \notin \mathcal{R}_c$.  It is important to point out that fixing $k \in \mathbb{N}^*$ and considering $m=l=k$, we have $L=2k\pi$, when $c=0$, and, from Theorem \ref{control-c}, it follows that \eqref{lpc} are not exactly controllable. Additionally, as mentioned before, we do not know if the system 
\begin{align}\label{NKdVN}
	\left\{
	\begin{array}{lll}
		u_t+(1+c)u_x+u_{xxx}+uu_x=0,&\text{ in }(0,L)\times (0,T),\\
		u_{xx}(0,t)=0,\ u_x(L,t)=h(t),\ u_{xx}(L,t)=0,&\text{ in }(0,T),\\
		u(x,0)=u_0(x),&\text{ in }(0,L),
	\end{array}
	\right.
\end{align}
is exactly controllable. So, in this context, the natural questions appear:

\vspace{0.2cm}
\noindent\textbf{Question $\mathcal{B}$:} Given $T>0$, $L \in \mathcal{R}_c$ and $u_0,u_T \in L^2(0,L)$ close enough to $c$, can one find an appropriate control input $h\in L^2(0,T)$ such that the solution $u$ of the system \eqref{NKdVN}, corresponding to $h$ and $u_0$, satisfies $u(\cdot,T)=u_T$? 
\vspace{0.1cm}

\vspace{0.2cm}
\noindent\textbf{Question $\mathcal{C}$:} Given $T>0$, $L \in \mathcal{R}_c$ and $u_0,u_T \in L^2(0,L)$. The system \eqref{NKdVN} is exactly controllable in the critical length $L$, that is, when $L \in \mathcal{R}_c$?
\vspace{0.1cm}

\subsection{Main results} 

Let us consider the nonlinear control system \eqref{NKdVN}  with $h$ as a control input and $u$ as the state. Theorem \ref{control} says that when $L\notin \mathcal{R}_c$, system \eqref{NKdVN} is locally controllable around $c$, but we do not know if the same holds when $L\in \mathcal{R}_c$.  The main result in this work provides an affirmative answer to the Question $\mathcal{C}$. Precisely, we have the following:

\vspace{0.2cm}

\begin{theorem}\label{c-length}
	Let $T>0$, $c=0$ and $L\in \mathcal{R}_0$. Then, system \eqref{NKdVN} is exactly controllable around the origin $0$ in $L^2(0,L)$, that is, there exists $\delta>0$ such that, for very $u_0,u_T\in L^2(0,L)$ with
	\begin{align*}
	\|u_0\|_{L^2(0,L)},\|u_T\|_{L^2(0,L)}<\delta,
	\end{align*}
it is possible to find $h\in L^2(0,T)$ such that the corresponding solution of \eqref{NKdVN} satisfying $u(\cdot,0)=u_0$ and $u(\cdot,T)=u_T$.
\end{theorem}

\vspace{0.2cm}

This previous result can be generalized for any $c$ as follows:

\vspace{0.2cm}

\begin{theorem}\label{c-length-a}
	Let $T>0$ and $L\in \mathcal{R}_c$. The system \eqref{NKdVN} is exactly controllable around $c$ in $L^2(0,L)$ in the sense of Theorem \ref{c-length}.
\end{theorem}

\vspace{0.2cm}

To prove the previous results we need an auxiliary property that ensures that for $c$ near enough to $0$ (small perturbations of $0$), the system \eqref{NKdV-2} is exactly controllable in a neighborhood of $c$ in $\mathcal{R}_c)$, or precisely,  for $d$ close enough to $c$ one has $L\notin \mathcal{R}_d$ so that, the system \eqref{lpc} corresponding to $d$ is exactly controllable, answering the Question $\mathcal{B}$. In other words, the set of critical lengths is sensitive to small disturbances in equilibrium $c$, and the result can be read as follows. 

\vspace{0.2cm}
\begin{theorem}\label{c-ap}
	Let $T>0$, $c\ne -1$ and $L \in \mathcal{R}_c$. There exists $\varepsilon_c>0$ such that, for every $d \in (c-\varepsilon_c,c+\varepsilon_c)\backslash\{c\}$, $d\ne -1$, we have $L \notin \mathcal{R}_d$. Consequently, the linear system \eqref{lpc}, with $c=d$, is exactly controllable; and the nonlinear system \eqref{NKdVN} is exactly controllable around the steady state $d$ in $L^2(0,L)$, that is, there exists $\delta_d>0$ such that, for any $u_0,u_T\in L^2(0,L)$ with
	\begin{align*}
		\|u_0-d\|_{L^2(0,L)}<\delta_d\text{ \ and \ }\|u_T-d\|_{L^2(0,L)}<\delta_d,
	\end{align*}
	one can find $h \in L^2(0,T)$ such that the system \eqref{NKdVN} admits a unique solution $u \in \mathcal{Z}_T$ satisfying
$u(\cdot,0)=u_0$ and $u(\cdot,T)=u_T$. 
\end{theorem}

\subsection{Heuristic and paper's outline}  

The proof of Theorem \ref{c-ap} is based on the topological properties of real numbers together with Theorem \ref{control}. Moreover, with this in hand, both results stated in the previous paragraph (Theorems \ref{c-length} and \ref{c-length-a}) rely on the so-called \textit{return method} together with the fixed point argument.

It is important to point out that the return method was introduced by J.M. Coron in \cite{Coron1992} (see also \cite{Coron1993}) and has been used by several authors to prove control results in the critical lengths for the KdV-type equation (see, for instance, \cite{cerpa1,cerpa2,CoCre,crepeau}). This method consists of building particular trajectories of the system \eqref{NKdVN} starting and ending at some equilibrium such that the linearization of the system around these trajectories has good properties. Here, we use a combination of this method with a fixed point argument, successfully applied in  \cite{Glass}.  We mention that this method can be applied together with quasi-static deformations and power series expansion. We refer the reader to the nice book of Coron \cite{CoronBook} for more details of the method.

It is important to emphasize that our control $h \in L^2(0,T)$ is guaranteed by the construction of $u$ (see below). In this setting, the resulting control (or trajectory) is not unique. Furthermore, choosing smoother initial and terminal data yields higher regularity for both the control and the state within this framework. However, as is well known in the literature, when the initial data lie in $L^2$, the regularity of the control is optimal at the $L^2$--level.

Concerning the construction of solutions to the Theorems \ref{c-length} and \ref{c-length-a}, we follow the following procedure: In the first step, we construct a solution that starts from $y_0$ and reaches at time $T / 3$ a state which is in some sense close to $d$ (which is yet to be defined). Then we construct a solution (close to the state solution $d$ ), which starts at time $2T / 3$ from the previous state. In the last step, we bring the latter state to 0 \textit{via} a function $u_2$, as we can see in Figure \ref{fig1} below, considering the characterization of the function $u:[0,L]\times [0,T]\rightarrow\mathbb{R}$ by
\begin{align*}
u=\left\{
\begin{array}{ll}
u_1,&\text{ in }[0,T/3],\\
d,&\text{ in }[T/3,2T/3],\\
u_2,&\text{ in }[2T/3,T].
\end{array}
\right.
\end{align*}

\begin{figure}[h!]
	\centering
	\includegraphics[width=0.8\linewidth]{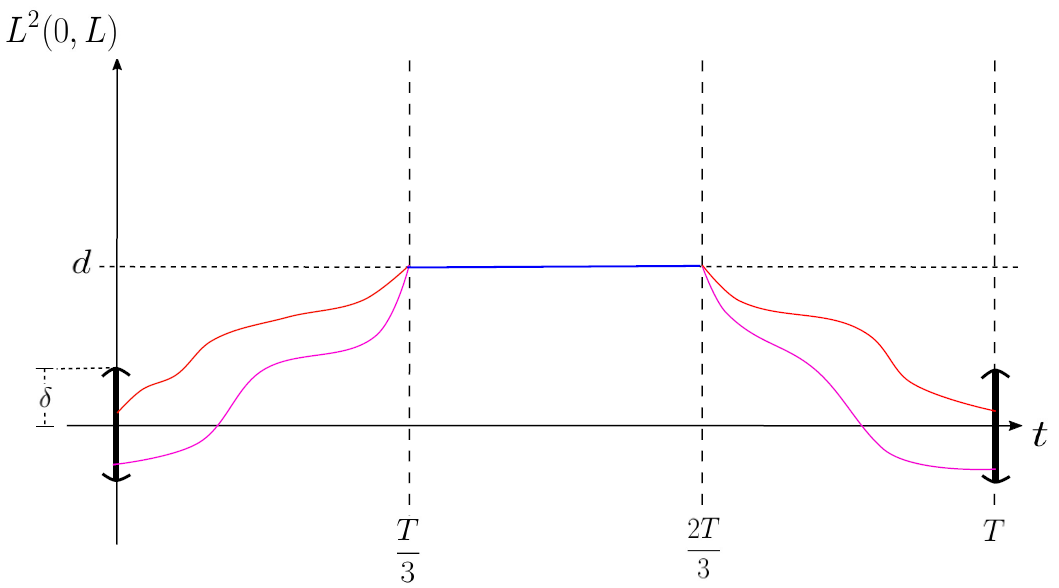}
	\caption{Solutions driving states close to 0 to constants and vice versa.}
	\label{fig1}
\end{figure}
 
We finish this introduction with an outline of this work, which consists of three parts, including the introduction. Section \ref{sec2} gives an overview of the well-posedness of the system \eqref{NKdVN}. Section \ref{sec3} is devoted to proving carefully the controllability of the system \eqref{NKdVN} when $L\in\mathcal{R}_c$. Precisely, in the first part of Section \ref{sec3}, we deal with the proof of Theorem \ref{c-ap}. In the second part, we prove the construction of the function $u$ mentioned before, and, in the third part of Section \ref{sec3}, we use these previous results to achieve Theorem \ref{c-length}. Finally, Section \ref{sec4} is devoted to final remarks and comments.

\section{Overview of the well-posedness theory}\label{sec2}

In this section, we review the well-posedness theory for the KdV equation. The results presented here can be found in \cite{BSZ,CaCaZh,KRZ}.  For that, consider $L>0$ and $T_0,T_1\in\mathbb{R}$ with $T_0<T_1$. We define the space
\begin{align*}
	\mathcal{Z}_{T_0,T_1}:=C\left([T_0,T_1];L^2(0,L)\right)\cap L^2\left([T_0,T_1];H^1(0,L)\right)
\end{align*}
which is a Banach space with the following norm
\begin{align*}
	\|y\|_{\mathcal{Z}_{T_0,T_1}}:=\max_{t \in [T_0,T_1]}\|y(\cdot,t)\|_{L^2(0,L)}+\left(\int_{T_0}^{T_1}\|y(\cdot,t)\|_{H^1(0,L)}^2dt\right)^{1/2}.
\end{align*}
For any $T>0$ we denote $\mathcal{Z}_{0,T}$ simply by $\mathcal{Z}_T$. 

Additionally, let $T>0$ be given and consider the space
\begin{align*}
	\mathcal{H}_T:=H^{-\frac{1}{3}}(0,T)\times L^2(0,T)\times H^{-\frac{1}{3}}(0,T)
\end{align*}
with a norm
$$
\|(h_1,h_2,h_3)\|_{\mathcal{H}_T}:=\|h_1\|_{H^{-\frac{1}{3}}(0,T)}+\|h_2\|_{L^2(0,T)}+\|h_3\|_{H^{-\frac{1}{3}}(0,T)}.
$$
The next proposition, showed in \cite[Proposition 2.5]{CaCaZh}, provides the well-posedness to the following system 
\begin{align}\label{wdrifit}
	\left\{
	\begin{array}{ll}
		u_t+u_{xxx}=f,&\ \text{ in }(0,L)\times(0,T),\\
		u_{xx}(0,t)=h_1(t),\ u_{x}(L,t)=h_2(t),\ u_{xx}(L,t)=h_3(t),&\ \text{ in }(0,T),\\
		u(x,0)=u_0,&\ \text{ in }(0,L).
	\end{array}
	\right.
\end{align}
\begin{proposition}[Caicedo, Capistrano-Filho, Zhang \cite{CaCaZh}]\label{wwdrifit}
For any $v_0 \in L^2(0,L)$, ${\bf h}=(h_1,h_2,h_3)\in \mathcal{H}_T$ and $f \in L^1(0,T,L^2(0,L))$, the IBVP \eqref{wdrifit} admits a unique mild solution $u \in \mathcal{Z}_T$, which satisfies
\begin{align*}
	\partial_x^ju\in L^\infty_x(0,L;H^{(1-j)/3}(0,T)),\ j=0,1,2.
\end{align*}
Moreover, there exists $C_1>0$ such that
\begin{align*}
	\|u\|_{\mathcal{Z}_T}+\sum_{j=0}^{2}\left\|\partial_x^ju\right\|_{L^\infty_x(0,L;H^{(1-j)/3}(0,T))}\leq C_1\left(\|u_0\|_{L^2(0,L)}+\|{\bf h}\|_{\mathcal{H}_T}+\|f\|_{L^1(0,T;L^2(0,L))}\right).
\end{align*}
\end{proposition}

\begin{remark}
We highlight that the constant $C_1$ in the above result depends on $T$. However, if $\theta \in (0,T]$ then, the estimates in the Proposition \ref{wwdrifit} hold with the same constant $C_1$ corresponding to $T$. 

\vspace{0.2cm}

In fact, let $u_0 \in L^2(0,L)$, ${\bf h}\in \mathcal{H}_\theta, \ f \in L^1(0,\theta;L^2(0,L))$ and $u\in \mathcal{Z}_\theta$ the solution of \eqref{wdrifit} corresponding to these datas. We extend ${\bf h}$ and $f$ to $[0,T]$ (we will also denote these extensions by ${\bf h}$ and $f$) putting
$${\bf h}=0 \text{ in }(\theta,T]\quad \text{and} \quad f=0\text{ in }(\theta,T].$$
Now, denote by $\tilde{u}\in \mathcal{Z}_T$ the corresponding solution of \eqref{wdrifit}. Then, from Proposition \ref{wwdrifit}, we have $\tilde{u}\big|_{[0,\theta]}=u$ and
\begin{align*}
\|u\|_{\mathcal{Z}_\theta}\leq \|\tilde{u}\|_{\mathcal{Z}_T}&\leq C_1\left(\|u_0\|_{L^2(0,L)}+\|{\bf h}\|_{\mathcal{H}_T}+\|f\|_{L^1(0,T;L^2(0,L))}\right)\\
&=C_1\left(\|u_0\|_{L^2(0,L)}+\|{\bf h}\|_{\mathcal{H}_\theta}+\|f\|_{L^1(0,\theta;L^2(0,L))}\right).
\end{align*}
\end{remark}

Using the Proposition \ref{wwdrifit}, we can get properties for the linear problem
\begin{align}\label{glp}
	\left\{
	\begin{array}{ll}
		y_t+(ay)_x+y_{xxx}=f,&\ \text{ in }(0,L)\times(0,T),\\
		y_{xx}(0,t)=h_1(t),\ y_{x}(L,t)=h_2(t),\ y_{xx}(L,t)=h_3(t),&\ \text{ in }(0,T),\\
		y(x,0)=y_0(x),&\ \text{ in }(0,L),
	\end{array}
	\right.
\end{align}
where $a \in \mathcal{Z}_T$ is given. To do this, the following lemma will be very useful and was proved in \cite[Lemma 3]{KZ}.

\vspace{0.2cm}

\begin{lemma}[Kramer, Zhang \cite{KZ}]\label{ywx} There exists a constant $C>0$ such that
	\begin{align*}
		\int_0^T\|uv_x(\cdot,t)\|_{L^2(0,L)}dt\leq C\left(T^\frac{1}{2}+T^\frac{1}{3}\right)\|u\|_{\mathcal{Z}_T}\|v\|_{\mathcal{Z}_T},
	\end{align*}
	for every $u,v\in \mathcal{Z}_T$.
\end{lemma}

\vspace{0.2cm}
With these previous results in hand, the following proposition gives us the well-posedness of the general system \eqref{glp}, which will be used several times, so, for the sake of completeness, we will give the proof. 

\vspace{0.2cm}
\begin{proposition}\label{wgproblem}
For any $y_0 \in L^2(0,L)$, ${\bf h}=(h_1,h_2,h_3)\in \mathcal{H}_T$ and  $f \in L^1(0,T,L^2(0,L))$, the IBVP \eqref{glp} admits a unique mild solution $y \in \mathcal{Z}_T$, which satisfies
\begin{align*}
	\partial_x^jy\in L^\infty_x(0,L;H^{(1-j)/3}(0,T)),\ j=0,1,2,
\end{align*}
and
\begin{align*}
	\|y\|_{\mathcal{Z}_T}\leq C_2\left(\|y_0\|_{L^2(0,L)}+\|{\bf h}\|_{\mathcal{H}_T}+\|f\|_{L^1(0,T;L^2(0,L))}\right),
\end{align*}
for some positive constant $C_2$ which depends only on $T$ and $\|a\|_{\mathcal{Z}_T}$. In addition, the solution $y$ possesses the following sharp trace estimates
\begin{align*}
	\sum_{j=0}^{2}\left\|\partial_x^jy\right\|_{L^\infty_x(0,L;H^{(1-j)/3}(0,T))}\leq C_2\left(\|y_0\|_{L^2(0,L)}+\|{\bf h}\|_{\mathcal{H}_T}+\|f\|_{L^1(0,T;L^2(0,L))}\right).
\end{align*}
In particular, the map $(y_0,{\bf h},f)\mapsto y$ is Lipschitz continuous.
\end{proposition}

\vspace{0.2cm}
\begin{proof}
Let $y_0\in L^2(0,L)$, ${\bf h}\in \mathcal{H}_T$ and $f \in L^1(0,T;L^2(0,L))$ be given. Consider $\theta$ satisfying $0<\theta\leq T$ and define the map $\Gamma:\mathcal{Z}_\theta\rightarrow\mathcal{Z}_\theta$ in the next way: for $y \in \mathcal{Z}_\theta$, put $\Gamma y$ being the solution of
\begin{align*}
	\left\{
	\begin{array}{ll}
		u_t+u_{xxx}=f-(ay)_x,&\ \text{ in }(0,L)\times(0,\theta),\\
		u_{xx}(0,t)=h_1(t),\ u_{x}(L,t)=h_2(t),\ u_{xx}(L,t)=h_3(t),&\ \text{ in }(0,\theta),\\
		u(x,0)=y_0(x),&\ \text{ in }(0,L).
	\end{array}
	\right.
\end{align*}
Consider the set
\begin{align*}
B=\left\{y \in \mathcal{Z}_\theta;\ \|y\|_{\mathcal{Z}_\theta}\leq r\right\},
\end{align*}
with $r>0$ to be determined later. From Proposition \ref{wwdrifit} and Lemma \ref{ywx} we have for any $y \in B$ the following estimate
\begin{equation}\label{eq 7}
\begin{aligned}
\|\Gamma y\|_{\mathcal{Z}_\theta}\leq&C_1\left(\|y_0\|_{L^2(0,L)}+\|{\bf h}\|_{\mathcal{H}_\theta}+\|f\|_{L^1(0,\theta;L^2(0,L))}+\|(ay)_x\|_{L^1(0,\theta;L^2(0,L))}\right)\\
\leq& C_1\left(\|y_0\|_{L^2(0,L)}+\|{\bf h}\|_{\mathcal{H}_\theta}+\|f\|_{L^1(0,\theta;L^2(0,L))}\right)\\&+2C_1C\left(\theta^\frac{1}{2}+\theta^\frac{1}{3}\right)\|a\|_{\mathcal{Z}_\theta}\|y\|_{\mathcal{Z}_\theta}\\
\leq& C_1\left(\|y_0\|_{L^2(0,L)}+\|{\bf h}\|_{\mathcal{H}_\theta}+\|f\|_{L^1(0,\theta;L^2(0,L))}\right)\\&+2C_1C\left(\theta^\frac{1}{2}+\theta^\frac{1}{3}\right)\|a\|_{\mathcal{Z}_T}\|y\|_{\mathcal{Z}_\theta}.
\end{aligned}
\end{equation}
Choosing
$$
r=2C_1\left(\|y_0\|_{L^2(0,L)}+\|{\bf h}\|_{\mathcal{H}_\theta}+\|f\|_{L^1(0,\theta;L^2(0,L))}\right)
$$
and $\theta$ satisfying
\begin{align}\label{eq 8}
2C_1C\left(\theta^\frac{1}{2}+\theta^\frac{1}{3}\right)\|a\|_{\mathcal{Z}_T}<\frac{1}{2},
\end{align}
the inequality \eqref{eq 7} give us
\begin{align*}
	\|\Gamma y\|_{\mathcal{Z}_\theta}\leq \frac{r}{2}+\frac{r}{2}=r,
\end{align*}
that is, $\Gamma(B)\subset B$. Furthermore, $\Gamma y-\Gamma w$ solves
\begin{align*}
	\left\{
	\begin{array}{ll}
		u_t+u_{xxx}=[a(-y+w)]_x,&\ \text{ in }(0,L)\times(0,\theta),\\
		u_{xx}(0,t)=u_{x}(L,t)=u_{xx}(L,t)=0,&\ \text{ in }(0,\theta),\\
		u(x,0)=0,&\ \text{ in }(0,L).
	\end{array}
	\right.
\end{align*}
So, from Proposition \ref{wwdrifit}, Lemma \ref{ywx} and inequality \eqref{eq 8}, we have
\begin{align*}
\|\Gamma y-\Gamma w\|_{\mathcal{Z}_\theta}&\leq C_1\|[a(y-w)]_x\|_{L^1(0,\theta;L^2(0,L))}\\&\leq
2C_1C\left(\theta^\frac{1}{2}+\theta^\frac{1}{3}\right)\|a\|_{\mathcal{Z}_\theta}\|y-w\|_{\mathcal{Z}_\theta}\\
&\leq 2C_1C\left(\theta^\frac{1}{2}+\theta^\frac{1}{3}\right)\|a\|_{\mathcal{Z}_T}\|y-w\|_{\mathcal{Z}_\theta}\\&<\frac{1}{2}\|y-w\|_{\mathcal{Z}_\theta}.
\end{align*}
Thus, $\Gamma:B\rightarrow B$ is a contraction so that, by Banach's fixed point theorem, $\Gamma$ has a fixed point $y\in B$ which is a solution to the problem \eqref{glp} in $[0,\theta]$, corresponding to data $(y_0,{\bf h},f)$. Additionally, inequalities \eqref{eq 7} and \eqref{eq 8} yields that \begin{align*}
	\|y\|_{\mathcal{Z}_\theta}\leq C_1\left(\|y_0\|_{L^2(0,L)}+\|{\bf h}\|_{\mathcal{H}_\theta}+\|f\|_{L^1(0,\theta;L^2(0,L))}\right)+\frac{1}{2}\|y\|_{\mathcal{Z}_\theta}
\end{align*}
and, therefore,
\begin{align*}
	\|y\|_{\mathcal{Z}_\theta}\leq 2C_1\left(\|y_0\|_{L^2(0,L)}+\|{\bf h}\|_{\mathcal{H}_\theta}+\|f\|_{L^1(0,\theta;L^2(0,L))}\right).
\end{align*}
Since $\theta$ depends only on $a$, with a standard continuation extension argument, the solution $y$ can be extended to interval $[0,T]$ and the following estimate holds
\begin{align}\label{eq 9}
	\|y\|_{\mathcal{Z}_T}\leq C_2\left(\|y_0\|_{L^2(0,L)}+\|{\bf h}\|_{\mathcal{H}_T}+\|f\|_{L^1(0,T;L^2(0,L))}\right),
\end{align}
for some suitable constant $C_2>0$ which only depends on $T$ and $\|a\|_{\mathcal{Z}_T}$. Therefore, it follows from \eqref{eq 9} that the map $(y_0,{\bf h},f)\mapsto y$ is Lipschitz continuous and, as a consequence of this, we have the uniqueness of the solution $y$ in $\mathcal{Z}_T$. The sharp trace estimates follow as a consequence of the Proposition \ref{wwdrifit}, showing the proposition. 
\end{proof}

Now, we will study the well-posedness of the nonlinear problem, namely:
\begin{align}\label{gnp}
\left\{
\begin{array}{ll}
y_t+(ay)_x+y_{xxx}+yy_x=f,&\ \text{ in }(0,L)\times(0,T),\\
y_{xx}(0,t)=h_1(t),\ y_{x}(L,t)=h_2(t),\ y_{xx}(L,t)=h_3(t),&\ \text{ in }(0,T),\\
y(x,0)=y_0(x),&\ \text{ in }(0,L).
\end{array}
\right.
\end{align}
The result can be read as follows. 

\vspace{0.2cm}

\begin{proposition}\label{wpgnp} Let $a \in \mathcal{Z}_T$ be given. For every $y_0 \in L^2(0,L)$, ${\bf h}=(h_1,h_2,h_3)\in \mathcal{H}_T$ and  $f \in L^1(0,T,L^2(0,L))$ there exists a unique mild solution $y\in \mathcal{Z}_T$ of \eqref{gnp} which satisfies
$$
	\|y\|_{\mathcal{Z}_T}\leq C_3\left(\|y_0\|_{L^2(0,L)}+\|{\bf h}\|_{L^2(0,T)}+\|f\|_{L^1(0,T;L^2(0,L))}\right)
$$
and
\begin{equation*}
\begin{aligned}
\sum_{j=0}^{2}\left\|\partial_x^jy\right\|_{L^\infty_x(0,L;H^{(1-j)/3}(0,T))}\leq& C_3\left(\|y_0\|_{L^2(0,L)}+\|h\|_{L^2(0,T)}~+\|f\|_{L^1(0,T;L^2(0,L))}\right)\\
&+C_3\left(\|y_0\|_{L^2(0,L)}+\|h\|_{L^2(0,T)}+\|f\|_{L^1(0,T;L^2(0,L))}\right)^2,
\end{aligned}
\end{equation*}
for some constant $C_3>0$. Furthermore, the corresponding solution map $S$ is locally Lipschitz continuous, that is, given $\lambda>0$ there exist $L_\lambda>0$ such that, for every
$$
(y_0,{\bf h},f_0),(y_1,{\bf g},f_1)\in L^2(0,L)\times \mathcal{H}_T\times L^1(0,T;L^2(0,L))
$$
with
$$
\begin{aligned}
\|y_0\|_{L^2(0,L)}+\|{\bf h}\|_{\mathcal{H}_T}+\|f_0\|_{L^1(0,T;L^2(0,L))}<\lambda,\\
\|y_1\|_{L^2(0,L)}+\|{\bf g}\|_{\mathcal{H}_T}+\|f_1\|_{L^1(0,T;L^2(0,L))}<\lambda,
\end{aligned}
$$
we have
\begin{equation*}
\begin{split}
\|S(y_0,{\bf h},f_0)-S(y_1,{\bf g},f_1)\|_{\mathcal{Z}_T}\leq& L_\lambda \left(\|y_0-y_1\|_{L^2(0,L)}\right.\\
&\left.+\|{\bf h}-{\bf g}\|_{\mathcal{H}_T}+\|f_0-f_1\|_{L^1(0,T;L^2(0,L))}\right).
\end{split}
\end{equation*}
\end{proposition}
\begin{proof} We will proceed as follows: First, we will show the existence of a solution and obtain the desired estimates. Secondly, we will get an estimate that provides the uniqueness of the solution and guarantees that $S$ is locally Lipschitz continuous. 

To do that, consider $y_0\in L^2(0,L)$, ${\bf h}\in \mathcal{H}_T$ and $f \in L^1(0,T;L^2(0,L))$. Let $\theta$ satisfying $0<\theta\leq T$ and define the map $\Gamma:\mathcal{Z}_\theta\rightarrow\mathcal{Z}_\theta$ in the following way: For $y \in \mathcal{Z}_\theta$, pick $\Gamma y$ as solution of the following problem
\begin{align*}
	\left\{
	\begin{array}{ll}
		u_t+(au)_x+u_{xxx}=f-yy_x,&\ \text{ in }(0,L)\times(0,\theta),\\
		u_{xx}(0,t)=h_1(t),\ u_{x}(L,t)=h_2(t),\ u_{xx}(L,t)=h_3(t),&\ \text{ in }(0,\theta),\\
		u(x,0)=y_0(x),&\ \text{ in }(0,L).
	\end{array}
	\right.
\end{align*}
Consider the set
\begin{align*}
	B=\left\{y \in \mathcal{Z}_\theta;\ \|y\|_{\mathcal{Z}_\theta}\leq r\right\},
\end{align*}
with $r>0$ to be determined later. Using Proposition \ref{wgproblem} and Lemma \ref{ywx} we obtain, for any $y \in B$,
\begin{equation}\label{eq 13}
\begin{aligned}
	\|\Gamma y\|_{\mathcal{Z}_\theta}&\leq C_2\left(\|y_0\|_{L^2(0,L)}+\|{\bf h}\|_{\mathcal{H}_\theta}+\|f\|_{L^1(0,\theta;L^2(0,L))}+\|yy_x\|_{L^1(0,\theta;L^2(0,L))}\right)\\
	&\leq C_2\left(\|y_0\|_{L^2(0,L)}+\|{\bf h}\|_{\mathcal{H}_\theta}+\|f\|_{L^1(0,\theta;L^2(0,L))}\right)+C_2C\left(\theta^\frac{1}{2}+\theta^\frac{1}{3}\right)\|y\|_{\mathcal{Z}_T}^2.\\
\end{aligned}
\end{equation}
Choosing
\begin{align*}
r=4C_2\left(\|y_0\|_{L^2(0,L)}+\|{\bf h}\|_{\mathcal{H}_\theta}+\|f\|_{L^1(0,\theta;L^2(0,L))}\right)
\end{align*}
and $\theta$ satisfying
\begin{align}\label{eq 14}
C_2C\left(\theta^\frac{1}{2}+\theta^\frac{1}{3}\right)r<\frac{1}{4},
\end{align}
inequality \eqref{eq 13} give us
\begin{align*}
\|\Gamma\|_{\mathcal{Z}_T}&\leq \frac{r}{4}+\frac{r}{4}=\frac{r}{2}<r,
\end{align*}
that is, $\Gamma(B)\subset B$. Moreover given $y,w\in B$, $\Gamma y-\Gamma w$ solves
\begin{align*}
	\left\{
	\begin{array}{ll}
		u_t+(au)_x+u_{xxx}=-yy_x+ww_x,&\ \text{ in }(0,L)\times(0,\theta),\\
		u_{xx}(0,t)=0,\ u_{x}(L,t)=0,\ u_{xx}(L,t)=0,&\ \text{ in }(0,\theta),\\
		u(x,0)=0,&\ \text{ in }(0,L).
	\end{array}
	\right.
\end{align*}
Therefore, from Proposition \ref{wgproblem},
\begin{align*}
	\|\Gamma y-\Gamma w\|_{\mathcal{Z}_\theta}\leq C_2\|yy_x-ww_x\|_{L^1(0,\theta;L^2(0,L))}.
\end{align*}
Note that
\begin{align*}
yy_x-ww_x=\frac{1}{2}\big[(y+w)_x(y-w)+(y+w)(y-w)_x\big].
\end{align*}
Then, thanks to the Lemma \ref{ywx}, we get that
\begin{align*}
&\|yy_x-ww_x\|_{L^1(0,\theta; L^2(0,L))}\\
&\leq \frac{1}{2}\big[\|(y+w)_x(y-w)\|_{L^1(0,\theta; L^2(0,L))}+\|(y+w)(y-w)_x\|_{L^1(0,\theta; L^2(0,L))}\big]\\
&\leq \frac{1}{2}\left[C\left(\theta^{\frac{1}{2}}+\theta^{\frac{1}{3}}\right)\|y+w\|_{\mathcal{Z}_\theta}\|y-w\|_{\mathcal{Z}_\theta}+C\left(\theta^{\frac{1}{2}}+\theta^{\frac{1}{3}}\right)\|y+w\|_{\mathcal{Z}_\theta}\|y-w\|_{\mathcal{Z}_\theta}\right]\\
&=C\left(\theta^{\frac{1}{2}}+\theta^{\frac{1}{3}}\right)\|y+w\|_{\mathcal{Z}_\theta}\|y-w\|_{\mathcal{Z}_\theta}\\
&\leq 2C\left(\theta^{\frac{1}{2}}+\theta^{\frac{1}{3}}\right)r\|y-w\|_{\mathcal{Z}_\theta}
\end{align*}
and, using \eqref{eq 14} yields
\begin{align*}
\|\Gamma y-\Gamma w\|_{\mathcal{Z}_\theta}\leq 2C_2C\left(\theta^{\frac{1}{2}}+\theta^{\frac{1}{3}}\right)r\|y-w\|_{\mathcal{Z}_\theta}<\frac{1}{2}\|y-w\|_{\mathcal{Z}_\theta}.
\end{align*}
Hence, $\Gamma:B\rightarrow B$ is a contraction so that, by Banach's fixed point theorem, $\Gamma$ has a fixed point $y\in B$ which is a solution to the problem \eqref{gnp} in $[0,\theta]$, corresponding to data $(y_0,{\bf h},f)$. 

Furthermore, due to the inequalities \eqref{eq 13} and \eqref{eq 14}, we have
\begin{align*}
	\|y\|_{\mathcal{Z}_\theta}\leq C_2\left(\|y_0\|_{L^2(0,L)}+\|{\bf h}\|_{\mathcal{H}_\theta}+\|f\|_{L^1(0,\theta;L^2(0,L))}\right)+\frac{1}{4}\|y\|_{\mathcal{Z}_\theta},
\end{align*}
and, therefore
\begin{align*}
	\|y\|_{\mathcal{Z}_\theta}\leq \frac{4}{3}C_2\left(\|y_0\|_{L^2(0,L)}+\|{\bf h}\|_{\mathcal{H}_\theta}+\|f\|_{L^1(0,\theta;L^2(0,L))}\right).
\end{align*}
Since $\theta$ depends only on $a$, with a standard continuation extension argument, the solution $y$ can be extended to interval $[0,T]$ and the following estimate holds
\begin{align}\label{eq 15}
	\|y\|_{\mathcal{Z}_T}\leq \tilde{C}_3\left(\|y_0\|_{L^2(0,L)}+\|{\bf h}\|_{\mathcal{H}_T}+\|f\|_{L^1(0,T;L^2(0,L))}\right),
\end{align}
for some suitable constant $\tilde{C}_3>0$ which only depends on $T$ and $a$. Now, using one more time the Proposition \ref{wpgnp} together with Lemma \ref{ywx} we obtain
\begin{align*}
&\sum_{j=0}^{2}\left\|\partial_x^jy\right\|_{L^\infty_x(0,L;H^{(1-j)/3}(0,T))}\\
&\leq C_2\left(\|y_0\|_{L^2(0,L)}+\|{\bf h}\|_{\mathcal{H}_T}+\|f\|_{L^1(0,T;L^2(0,L))}+\|yy_x\|_{L^1(0,T;L^2(0,L))}\right)\\
&\leq  C_2\left(\|y_0\|_{L^2(0,L)}+\|{\bf h}\|_{\mathcal{H}_T}+\|f\|_{L^1(0,T;L^2(0,L))}\right)+C_2C\left(\theta^\frac{1}{2}+\theta^\frac{1}{3}\right)\|y\|_{\mathcal{Z}_T}^2.
\end{align*}
By \eqref{eq 15} it follows that
\begin{align*}
&\sum_{j=0}^{2}\left\|\partial_x^jy\right\|_{L^\infty_x(0,L;H^{(1-j)/3}(0,T))}\\
&\leq C_2\left(\|y_0\|_{L^2(0,L)}+\|{\bf h}\|_{\mathcal{H}_T}+\|f\|_{L^1(0,T;L^2(0,L))}\right)\\
&+C_2C\left(\theta^\frac{1}{2}+\theta^\frac{1}{3}\right)\tilde{C}^2\left(\|y_0\|_{L^2(0,L)}+\|{\bf h}\|_{\mathcal{H}_T}+\|f\|_{L^1(0,T;L^2(0,L))}\right)^2.
\end{align*}
Choosing
\begin{align*}
C_3:=\max\left\{\tilde{C}_3, C_2, C_2C\left(\theta^\frac{1}{2}+\theta^\frac{1}{3}\right)\tilde{C}^2\right\}
\end{align*}
we obtain the desired estimates.

Now, consider $$
(y_0,{\bf h},f_0),(y_1,{\bf g},f_1)\in L^2(0,L)\times \mathcal{H}_T\times L^1(0,T;L^2(0,L))
$$
with
$$
\begin{aligned}
	\|y_0\|_{L^2(0,L)}+\|{\bf h}\|_{\mathcal{H}_T}+\|f_0\|_{L^1(0,T;L^2(0,L))}<\lambda,\\
	\|y_1\|_{L^2(0,L)}+\|{\bf g}\|_{\mathcal{H}_T}+\|f_1\|_{L^1(0,T;L^2(0,L))}<\lambda.
\end{aligned}
$$
Write ${\bf h}=(h_1,h_2,h_3)$ and ${\bf g}=(g_1,g_2,g_3)$. Let $y$ and $u$ be solutions of
\begin{align*}
	\left\{
	\begin{array}{ll}
		y_t+(ay)_x+y_{xxx}+yy_x=f_0,&\ \text{ in }(0,L)\times(0,T),\\
		y_{xx}(0,t)=h_1(t),\ y_{x}(L,t)=h_2(t),\ y_{xx}(L,t)=h_3(t),&\ \text{ in }(0,T),\\
		y(x,0)=y_0(x)&\ \text{ in }(0,L),
	\end{array}
	\right.
\end{align*}
and
\begin{align*}
	\left\{
	\begin{array}{ll}
		u_t+(au)_x+u_{xxx}+uu_x=f_1,&\ \text{ in }(0,L)\times(0,T),\\
		u_{xx}(0,t)=g_1(t),\ u_{x}(L,t)=g_2(t),\ u_{xx}(L,t)=g_3(t),&\ \text{ in }(0,T),\\
		u(x,0)=y_1(x)&\ \text{ in }(0,L),
	\end{array}
	\right.
\end{align*}
respectively. Then, $w=y-u$ solves the problem
\begin{align*}
	\left\{
	\begin{array}{ll}
		w_t+\big[\left(a+\frac{1}{2}(y+u)\right)w\big]_x+w_{xxx}=f_0-f_1,&\ \text{ in }(0,L)\times(0,T,)\\
		w_{xx}(0,\cdot)=h_1-g_1,\ w_{x}(L,\cdot)=h_2-g_2,\ w_{xx}(L,\cdot)=h_3-g_3,&\ \text{ in }(0,T),\\
		w(x,0)=y_0(x)-y_1(x),&\ \text{ in }(0,L).
	\end{array}
	\right.
\end{align*}
From Proposition \ref{wgproblem} it follows that
\begin{align}\label{eq 16}
\|y-u\|_{\mathcal{Z}_T}\leq D\left(\|y_0-y_1\|_{L^2(0,L)}+\|{\bf h}-{\bf g}\|_{\mathcal{H}_T}+\|f_0-f_1\|_{L^1(0,T;L^2(0,L))}\right),
\end{align}
where $D$ is a constant which depends on $T$ and $\|a+\frac{1}{2}(y+u)\|_{\mathcal{Z}_T}$. But, using \eqref{eq 15} we have
\begin{align*}
\|a+\frac{1}{2}(y+u)\|_{\mathcal{Z}_T}\leq& \|a\|_{\mathcal{Z}_T}+\frac{1}{2}\left(\|y\|_{\mathcal{Z}_T}+\|u\|_{\mathcal{Z}_T}\right)\\
\leq& \|a\|_{\mathcal{Z}_T}+\frac{\tilde{C}_3}{2}\left(\|y_0\|_{L^2(0,L)}+\|{\bf h}\|_{\mathcal{H}_T}+\|f_0\|_{L^1(0,T;L^2(0,L))}\right)\\
&+\frac{\tilde{C}_3}{2}\left(\|y_1\|_{L^2(0,L)}+\|{\bf g}\|_{\mathcal{H}_T}+\|f_1\|_{L^1(0,T;L^2(0,L))}\right)
\end{align*}
and, consequently,
\begin{align*}
\|a+\frac{1}{2}(y+u)\|_{\mathcal{Z}_T}&\leq \|a\|_{\mathcal{Z}_T}+\tilde{C}_3\lambda.
\end{align*}
Hence, the constant $D$ can be chosen depending only on $T$ and $\|a\|_{\mathcal{Z}_T}$ (and also on $\lambda$), so that \eqref{eq 16} gives us the uniqueness of the solution, which turns the solution map $S$ well-defined. Moreover, from \eqref{eq 16}
\begin{equation*}
\begin{split}
\|S(y_0,{\bf h},f_0)-S(y_1,{\bf g},f_1)\|_{\mathcal{Z}_T}\leq &D \left(\|y_0-y_1\|_{L^2(0,L)}\right.\\&\left.+\|{\bf h}-{\bf g}\|_{\mathcal{H}_T}+\|f_0-f_1\|_{L^1(0,T;L^2(0,L))}\right),
\end{split}
\end{equation*}
which concludes the proof.
\end{proof}

Finally, the following Lemma, whose proof can be found in \cite{Rosier}, will be useful in the next section.

\vspace{0.2cm}

\begin{lemma}[Rosier \cite{Rosier}]\label{yyx}
	If $y\in \mathcal{Z}_T$ then $yy_x \in L^1(0,T;L^2(0,L))$ and the map
	\begin{align*}
		\begin{array}{rcl}
			\mathcal{Z}_T&\longrightarrow&L^1(0,T;L^2(0,L))\\
			y&\longmapsto&yy_x,
		\end{array}
	\end{align*}
	is continuous. More precisely, for every $y,z\in \mathcal{Z}_T$ we have that
\begin{equation*}
\begin{split}
		\|yy_x-zz_x\|_{L^1(0,T;L^2(0,L))}\leq &C_4\left(\|y\|_{L^2(0,T;H^1(0,L))}\right.\\&\left.+\|z\|_{L^2(0,T;H^1(0,L))}\right)\|y-z\|_{L^2(0,T;H^1(0,L))},
\end{split}
\end{equation*}
	where $C_4$ is a positive constant that depends only on $L$.
\end{lemma}

\vspace{0.2cm}

\begin{remark}\label{mia}We end this subsection with the following remarks.
\begin{enumerate}
	\item For every $a\in \mathcal{Z}_T$, the Proposition \ref{wgproblem} give us the well-definition for the solution operator $$\Lambda_a:L^2(0,L)\times \mathcal{H}_T\times L^1(0,T;L^2(0,L))\rightarrow\mathcal{Z}_T$$
	where, for each $(y_0,{\bf h},f)\in L^2(0,L)\times \mathcal{H}_T\times L^1(0,T;L^2(0,L))$, $\Lambda_a(y_0,{\bf h},f)$ is the corresponding solution to the problem \eqref{glp}. Furthermore, the Proposition \ref{wpgnp} guarantees that $\Lambda_a$ is a bounded linear operator.
	\item Of course, the constant $C_2>0$ in the Proposition \ref{wgproblem} depends on $\|a\|_{\mathcal{Z}_T}$. However, given $M>0$, the same constant $C_2$ can be used for every $a\in \mathcal{Z}_T$ with $\|a\|_{\mathcal{Z}_T}\leq M$.
\end{enumerate}
\end{remark}

\section{Boundary controllability in the critical length}\label{sec3}

In this section, we study the controllability of the system
\begin{align}\label{ncp}
	\left\{
	\begin{array}{ll}
		y_t+y_x+y_{xxx}+yy_x=0,&\ \text{ in }(0,L)\times(0,T),\\
		y_{xx}(0,t)=0,\ y_{x}(L,t)=h(t),\ y_{xx}(L,t)=0,&\ \text{ in }(0,T),\\
		y(x,0)=y_0(x)&\ \text{ in }(0,L),
	\end{array}
	\right.
\end{align}
when $L$ is a critical length. We will use the return method together with the fixed point argument to ensure the controllability of the system \eqref{ncp} when $L \in \mathcal{R}_c$. Before presenting the proof of the main result, let us give a preliminary result that is important in our analysis.

\subsection{An auxiliary result} 

Given $c\ne -1$, the set of critical lengths for the linearization of \eqref{ncp} around $c$ is
\begin{align*}
\mathcal{R}_c:=\left\{\frac{2\pi}{\sqrt{3(c+1)}}\sqrt{m^2+ml+l^2}\ ;\ m,l\in \mathbb{N}^*\right\}\cup\left\{\frac{m\pi}{\sqrt{c+1}};\ m\in \mathbb{N}^*\right\}.
\end{align*}
As mentioned before, when $L \in \mathcal{R}_c$, the linear system
\begin{align}\label{lac}
	\left\{
	\begin{array}{lll}
		y_t+(1+c)y_x+y_{xxx}=0,&\text{ in }(0,L)\times (0,T),\\
		y_{xx}(0,t)=0,\ y_x(L,t)=h(t),\ y_{xx}(L,t)=0,&\text{ in }(0,T),\\
		y(x,0)=y_0(x),&\text{ in }(0,L),
	\end{array}
	\right.
\end{align}
it is not exactly controllable. So, in this section, we first give the proof of Theorem \ref{c-ap}. 

\vspace{0.2cm}

\begin{proof}\textbf{(Proof of Theorem \ref{c-ap}.)}
We will split the proof into two cases.

\vspace{0.2cm}

\noindent{\textbf{First case:}} $L=2\pi\frac{\sqrt{j^2+jk+k^2}}{\sqrt{3(c+1)}}$ for some $(j,k) \in \mathbb{N}^*\times \mathbb{N}^*$.

\vspace{0.2cm}

Let $d\ne -1$ and suppose that $L\in \mathcal{R}_d$. Then
\begin{equation}\label{eq-10}
	L=2\pi\frac{\sqrt{m^2+ml+l^2}}{\sqrt{3(d+1)}};\ m,l\in \mathbb{N}^*
\end{equation}
or
\begin{equation}\label{eq-11}
	L=\frac{m\pi}{\sqrt{d+1}};\ m \in \mathbb{N}^*.
\end{equation}

If \eqref{eq-10} is the case,  we have that 
$$
2\pi\frac{\sqrt{j^2+jk+k^2}}{\sqrt{3(c+1)}}=2\pi\frac{\sqrt{m^2+ml+l^2}}{\sqrt{3(d+1)}},
$$
which implies
$$
d=\frac{(c+1)(m^2+ml+l^2)}{j^2+jk+k^2}-1
$$
with $m,l \in \mathbb{N}^*$. 

Otherwise, if \eqref{eq-11} holds, so
$$
2\pi\frac{\sqrt{j^2+jk+k^2}}{\sqrt{3(c+1)}}=\frac{m\pi}{\sqrt{d+1}}
$$
giving that
$$
d=\frac{3m^2(c+1)}{4(j^2+jk+k^2)}-1
$$
where $m\in \mathbb{N}^*$. Therefore, if $L\in \mathcal{R}_d$ then we necessarily have $d \in \mathcal{A}_1\cup \mathcal{B}_1$. Here, 
$$
\mathcal{A}_1:=\left\{
\frac{(c+1)(m^2+ml+l^2)}{j^2+jk+k^2}-1,\ m,l\in \mathbb{N}^*
\right\}
$$
and
$$\mathcal{B}_1:=\left\{\frac{3m^2(c+1)}{4(j^2+jk+k^2)}-1;\ m \in \mathbb{N}^*\right\}.
$$

We are now in a position to prove that $\mathcal{A}_1\cup\mathcal{B}_1$ is discrete. To do that, consider $x,y \in \mathcal{A}_1$ such that $x\ne y$ in the form
$$
x=\frac{(c+1)(m_1^2+m_1l_1+l_1^2)}{j^2+jk+k^2}-1
$$
and 
$$ y=\frac{(c+1)(m_2^2+m_2l_2+l_2^2)}{j^2+jk+k^2}-1,
$$
with $m_1,l_1,m_2,l_2\in \mathbb{N}^*$. Note that
$$
x-y=\frac{c+1}{j^2+jk+k^2}\left[(m_1^2+m_1l_1+l_1^2)-(m_2^2+m_2l_2+l_2^2)\right].
$$
Since $x\ne y$ we have
$$
(m_1^2+m_1l_1+l_1^2)\ne (m_2^2+m_2l_2+l_2^2).
$$
Thus
$$
\left|(m_1^2+m_1l_1+l_1^2)-(m_2^2+m_2l_2+l_2^2)\right|\geq 1
$$
so that
$$
d(x,y)\geq \frac{c+1}{j^2+jk+k^2},\ \forall x,y \in \mathcal{A}_1,\ x\ne y.
$$
Analogously, we get 
$$
d(x,y)\geq \frac{3(c+1)}{4(j^2+jk+k^2)},\ \forall\ x,y\in \mathcal{B}_1, x\ne y.
$$
Now, let $x \in \mathcal{A}_1$ and $y\in \mathcal{B}_1$ with $x\ne y$, as follow:
$$
x=\frac{(c+1)(m^2+ml+l^2)}{j^2+jk+k^2}-1
$$
and 
$$y=\frac{3p^2(c+1)}{4(j^2+jk+k^2)}-1,$$ 
where $m,l,p\in \mathbb{N}^*$. Observe that
\begin{equation*}
	\begin{split}
		x-y=&\frac{(c+1)(m^2+ml+l^2)}{j^2+jk+k^2}-\frac{3p^2(c+1)}{4(j^2+jk+k^2)}\\
		=&\frac{4(c+1)(m^2+ml+l^2)}{4(j^2+jk+k^2)}-\frac{3p^2(c+1)}{4(j^2+jk+k^2)}\\
		=&\frac{c+1}{4(j^2+jk+k^2)}\left[4(m^2+ml+l^2)-3p^2\right].
	\end{split}
\end{equation*}
Since $x\ne y$ we have that $4(m^2+ml+l^2)$ and $3p^2$ are distinct natural numbers so
$$
\left|4(m^2+ml+l^2)-3p^2\right|\geq 1
$$
and, consequently, we have
$$
d(x,y)\geq \frac{c+1}{4(j^2+jk+k^2)}.
$$
From all the above, we conclude that
$$
d(x,y)\geq \frac{c+1}{4(j^2+jk+k^2)},\ \forall\ x,y\in \mathcal{A}_1\cup \mathcal{B}_1,\ x\ne y,
$$
which implies that $\mathcal{A}_1\cup\mathcal{B}_1$ is discrete.

Note that $c \in \mathcal{A}_1\cup \mathcal{B}_1$ since
$$
c=\left[\frac{(c+1)(j^2+jk+k^2)}{j^2+jk+k^2}-1\right]\in \mathcal{A}_1.
$$
As any point in $\mathcal{A}_1\cup \mathcal{B}_1$ is isolated, there exists $\epsilon_1>0$ such that
$$
(c-\epsilon_1,c+\epsilon_1)\cap \left[\mathcal{A}_1\cup \mathcal{B}_1\right]=\{c\}.
$$
Therefore for $d \in (c-\epsilon_1,c+\epsilon_1)\backslash\{c\}$ we have that $d\notin\mathcal{A}_1\cup \mathcal{B}_1$ and, therefore, $L \notin \mathcal{R}_d$.

\vspace{0.2cm}

\noindent{\textbf{Second case:}} $L=\frac{k\pi}{\sqrt{c+1}}$ for some $k \in \mathbb{N}^*$.

\vspace{0.2cm}

Let $d\ne -1$ and suppose that $L\in \mathcal{R}_d$, that is, \eqref{eq-10} or \eqref{eq-11} holds. If \eqref{eq-10} holds, then
$$\frac{k\pi}{\sqrt{c+1}}=2\pi\frac{\sqrt{m^2+ml+l^2}}{\sqrt{3(d+1)}}$$ giving that
$$
d=\frac{(c+1)(m^2+ml+l^2)}{3k^2}-1,
$$
with $m,l\in \mathbb{N}^*$. If \eqref{eq-11} is the case, thus
$$
\frac{k\pi}{\sqrt{c+1}}=\frac{m\pi}{\sqrt{d+1}}
$$
so that
$$
d=\frac{(c+1)m^2}{k^2}-1.
$$
Therefore, if $L\in \mathcal{R}_d$ then we have necessarily  $d \in \mathcal{A}_2\cup \mathcal{B}_2$ where
$$
\mathcal{A}_2:=\left\{
\frac{(c+1)(m^2+ml+l^2)}{3k^2}-1;\ m,l\in \mathbb{N}^*
\right\}
$$
and
$$
\mathcal{B}_2:=\left\{\frac{(c+1)m^2}{k^2}-1;\ m \in \mathbb{N}^*\right\}.
$$

As done before, taking $x,y \in \mathcal{A}_2$ with $x\ne y$, such that 
$$x=\frac{(c+1)(m_1^2+m_1l_1+l_1^2)}{3k^2}-1$$
and $$y=\frac{(c+1)(m_2^2+m_2l_2+l_2^2)}{3k^2}-1,$$
where $m_1,l_1,m_2,l_2\in \mathbb{N}^*$,yields that 
$$
x-y=\frac{c+1}{3k^2}\left[(m_1^2+m_1l_1+l_1^2)-(m_2^2+m_2l_2+l_2^2)\right]
$$
and, as $x\ne y$ we have
$$\left|(m_1^2+m_1l_1+l_1^2)-(m_2^2+m_2l_2+l_2^2)\right|\geq 1.$$
Consequently,
$$
d(x,y)\geq \frac{c+1}{3k^2}, \forall x,y \in \mathcal{A}_2,\ x\ne y.
$$
In an analogous way, 
$$
d(x,y)\geq \frac{c+1}{k^2},\ \forall x,y \in \mathcal{B}_2,\ x\ne y
$$
and
$$
d(x,y)\geq \frac{c+1}{3k^2},\ \forall x\in \mathcal{A}_2,\ y\in \mathcal{B}_2,\ x\ne y.
$$
Therefore, 
$$
d(x,y)\geq \frac{c+1}{3k^2},\ \forall x,y \in \mathcal{A}_2\cup \mathcal{B}_2,\ x\ne y.
$$

Proceeding as in the first case, we conclude that there exists $\epsilon_2>0$ such that, for every $d \in (c-\epsilon_2,c+\epsilon_2)\backslash\{c\}$ we have $d\notin \mathcal{A}_2\cup\mathcal{B}_2$ so that $L \notin \mathcal{R}_d$. Considering $\epsilon_c:=\min\{\epsilon_1,\epsilon_2\}$, thanks thanks to the \cite[Proposition 3.6]{CaCaZh} and Theorem \ref{control}, the proof is completed. 
\end{proof}

\begin{remark}
It can seen that, for $c\ne -1$ and $L \in \mathcal{R}_c$, we must have $0<\epsilon_c\leq \frac{\pi^2}{3L^2}$. Hence, $\epsilon_c$ does not depend on the control time $T$ and $\epsilon_c\rightarrow 0$ when $L\rightarrow \infty$. On the other hand, for $L$ close enough to 0, the tolerance $\epsilon_c$ can be considered extremely large.
\end{remark}

\subsection{Construction of the trajectories} 

In this subsection, for simplicity, we will consider  $c=0$ in $\mathcal{R}_c$. Let $T>0$ and note that by Theorem \ref{c-ap} there exist $\epsilon_0>0$ such that for every $d\in (0,\epsilon_0)$, the system \eqref{lac} is exactly controllable around $d$. Henceforth, we use $C,C_1,C_2$ and $C_3$ to denote the positive constants given in Proposition \ref{wpgnp}, corresponding to $T$. One can see that, for every $\tau \in [0,T]$, the same constants can be used to apply the corresponding results for $\|\cdot\|_{\mathcal{Z}_\tau}$ so, for simplicity, we will consider now on $\tau=\frac{T}{3}$. 

The first result of this section ensures that we can construct solutions for the system \eqref{ncp} which starts close to $0$ (left-hand side of the spatial domain) and achieves some non-null equilibrium in a certain time $T/3$. The result is shown by a fixed-point argument.

\vspace{0.2cm}

\begin{proposition}\label{d-0to-c}
	There exist $\delta_1>0$ such that, for every $d \in (0,\delta_1)$ and  $y_0\in L^2(0,L)$ with $\|y_0\|_{L^2(0,L)}<\delta_{1},$
	there exists $h_1 \in L^2(0,T/3)$ such that, the solution of \eqref{ncp} for $t \in [0,T/3]$, satisfies
	$$
	y(\cdot,T/3)=d.
	$$
\end{proposition}

\begin{proof} Let $\delta_{1}\in (0,\epsilon_0)$ be a number to be chosen later. Consider $d \in (0,\delta_{1})$ and $y_0\in L^2(0,L)$ satisfying
$$
	\|y_0\|_{L^2(0,L)}<\delta_1.
$$
For $\varepsilon\in (0,\epsilon_0)$ such that
\begin{align}\label{estimate to epsilon}
C\left(\tau^\frac{1}{2}+\tau^\frac{1}{3}\right)\|\varepsilon\|_{\mathcal{Z}_T}<\delta_1\text{ \ \ and \ \ }C\left(\tau^{\frac{1}{2}}+\tau^{\frac{1}{3}}\right)\|\varepsilon\|_{\mathcal{Z}_T}^2<\delta_1,
\end{align}
where $C>0$ is the positive constant given in Lemma \ref{ywx}, \cite[Proposition 3.6]{CaCaZh} guarantees the existence of a bounded linear control operator
$$
\Psi^\varepsilon:L^2(0,L)\times L^2(0,L)\rightarrow L^2(0,\tau)
$$
such that, for any $u_0,u_\tau \in L^2(0,L)$, the solution $u$ of
\begin{align*}
	\left\{
	\begin{array}{lll}
		u_t+(1+\varepsilon)u_x+u_{xxx}=0,&\text{ in }(0,L)\times (0,\tau),\\
		u_{xx}(0,t)=0,\ 	u_x(L,t)=h(t),\ u_{xx}(L,t)=0,&\text{ in }(0,\tau),\\
		u(x,0)=u_0(x),&\text{ in }(0,L),
	\end{array}
	\right.
\end{align*}
with the control $h=\Psi^\varepsilon(u_0,u_\tau)$ satisfies $u(\cdot, \tau)=u_\tau$.

We will denote, for simplicity, the solution operator $\Lambda_{1+\varepsilon}$ (given in Proposition \ref{wgproblem} and Remark \ref{mia} with $a=1+\varepsilon$) by $\Lambda_\varepsilon$. Observe that, if $y$ is solution for \eqref{ncp} for some control $h$, then $y$ is a solution of
\begin{align*}
	\left\{
	\begin{array}{ll}
		y_t+(1+\varepsilon)y_x+y_{xxx}=-yy_x+\varepsilon y_x,&\ \text{ in }(0,L)\times(0,\tau),\\
		y_{xx}(0,t)=0,\ y_{x}(L,t)=h(t),\ y_{xx}(L,t)=0,&\ \text{ in }(0,\tau),\\
		y(x,0)=y_0(x)&\ \text{ in }(0,L),
	\end{array}
	\right.
\end{align*}
that is,
\begin{align}\label{ex-nonlinear}
y=\Lambda_\varepsilon(y_0,h,-yy_x+\varepsilon y_x)=\Lambda_\varepsilon(y_0,h,0)+\Lambda_\varepsilon(0,0,-yy_x+\varepsilon y_x).
\end{align}

Let $y\in \mathcal{Z}_\tau$ and $h_y\in L^2(0,\tau)$ given by
$$
h_y=\Psi^\varepsilon\big(y_0,d-\Lambda_\varepsilon(0,0,-yy_x+\varepsilon y_x)(\cdot,\tau)\big).
$$
Define the map $\Gamma:\mathcal{Z}_\tau\rightarrow\mathcal{Z}_\tau$ by
\begin{align*}
\Gamma y=\Lambda_\varepsilon(y_0,h_y,0)+\Lambda_\varepsilon(0,0,-yy_x+\varepsilon y_x).
\end{align*}
Note that if $\Gamma$ has a fixed point $y$, then, from the above construction, it follows that $y$ is a solution of \eqref{ncp} with the control $h_y$. Moreover, from \eqref{ex-nonlinear} we have
\begin{align*}
	y=\Lambda_\varepsilon(y_0,h_y,0)+\Lambda_\varepsilon(0,0,-yy_x+\varepsilon y_x)
\end{align*}
so, by definitions of $\Lambda_\varepsilon,\ h_y$ and $\Psi^\varepsilon$, we get that
\begin{align*}
y(\cdot,\tau)&=\Lambda_\varepsilon(y_0,h_y,0)(\cdot,\tau)+\Lambda_\varepsilon(0,0,-yy_x+\varepsilon y_x)(\cdot,\tau)\\
&=d-\Lambda_\varepsilon(0,0,-yy_x+\varepsilon y_x)(\cdot,\tau)+\Lambda_\varepsilon(0,0,-yy_x+\varepsilon y_x)(\cdot,\tau)\\
&=d,
\end{align*}
and our problem would be solved. So we will focus our efforts on showing that $\Gamma$ has a fixed point in a suitable metric space.

To do that, let  $B$ be the set
\begin{align*}
B=\left\{y \in \mathcal{Z}_\tau;\ \|y\|_{\mathcal{Z}_\tau}\leq r\right\},
\end{align*}
with $r>0$ to be chosen later. By \eqref{ex-nonlinear} and Proposition \ref{wgproblem} (together with Remark \ref{mia}) we have, for $y \in B$, that
\begin{equation*}
\begin{split}
\|\Gamma y\|_{\mathcal{Z}_\tau}&\leq \|\Lambda_\varepsilon(y_0,h_y,0)\|_{\mathcal{Z}_\tau}+\|\Lambda_\varepsilon(0,0,-yy_x+\varepsilon y_x)\|_{\mathcal{Z}_\tau}\\
&\leq C_2\left(\|y_0\|_{L^2(0,L)}+\|h_y\|_{L^2(0,\tau)}+\|-yy_x+\varepsilon y_x\|_{L^1(0,\tau;L^2(0,L))}\right).
\end{split}
\end{equation*}
From Lemma \ref{ywx} and Young's inequality, we ensure that
\begin{align*}
\|yy_x-\varepsilon y_x\|_{L^1(0,\tau,L^2(0,L))}&=\|(y-\varepsilon)y_x\|_{L^1(0,\tau,L^2(0,L))}\\&\leq C\left(\tau^\frac{1}{2}+\tau^\frac{1}{3}\right)\|y-\varepsilon\|_{\mathcal{Z}_\tau}\|y\|_{\mathcal{Z}_\tau}\\
&\leq C\left(\tau^\frac{1}{2}+\tau^\frac{1}{3}\right)\frac{1}{2}\left(\|y-\varepsilon\|_{\mathcal{Z}_\tau}^2+\|y\|_{\mathcal{Z}_\tau}^2\right)\\
&\leq C\left(\tau^\frac{1}{2}+\tau^\frac{1}{3}\right)\frac{1}{2}\left(\|y\|_{\mathcal{Z}_\tau}^2+\|y\|_{\mathcal{Z}_\tau}^2+\|\varepsilon\|_{\mathcal{Z}_\tau}^2+\|\varepsilon\|_{\mathcal{Z}_\tau}^2+\|y\|_{\mathcal{Z}_\tau}^2\right).
\end{align*}
Thus,
\begin{align*}
\|yy_x-\varepsilon y_x\|_{L^1(0,\tau,L^2(0,L))}&\leq C\left(\tau^\frac{1}{2}+\tau^\frac{1}{3}\right)\frac{3}{2}\|y\|_{\mathcal{Z}_\tau}^2+C\left(\tau^\frac{1}{2}+\tau^\frac{1}{3}\right)\|\varepsilon\|_{\mathcal{Z}_\tau}^2
\end{align*}
and, by \eqref{estimate to epsilon} we obtain
\begin{align*}
\|yy_x-\varepsilon y_x\|_{L^1(0,\tau,L^2(0,L))}&\leq C\left(\tau^\frac{1}{2}+\tau^\frac{1}{3}\right)\frac{3}{2}\|y\|_{\mathcal{Z}_\tau}^2+\delta_1,
\end{align*}
that is,
\begin{align*}
	\|yy_x-\varepsilon y_x\|_{L^1(0,\tau,L^2(0,L))}&\leq\overline{C}\|y\|_{\mathcal{Z}_\tau}^2+\delta_1,
\end{align*}
where
\begin{align*}
\overline{C}:=\frac{3C}{2}\left(\tau^\frac{1}{2}+\tau^\frac{1}{3}\right).
\end{align*}
In this way
\begin{align*}
\|h_y\|_{L^2(0,\tau)}&=\|\Psi^\varepsilon\big(y_0,d-\Lambda_\varepsilon(0,0,-yy_x+\varepsilon y_x)(\cdot,\tau)\big)\|_{L^2(0,\tau)}\\
&\leq \|\Psi^\varepsilon\|\left(\|y_0\|_{L^2(0,L)}+\|d\|_{L^2(0,L)}+\|\Lambda_\varepsilon(0,0,-yy_x+\varepsilon y_x)(\cdot,\tau)\|_{L^2(0,L)}\right)\\
&\leq \|\Psi^\varepsilon\|\delta_1+\|\Psi^\varepsilon\|d\sqrt{L}+\|\Psi^\varepsilon\|\|\Lambda_\varepsilon(0,0,-yy_x+\varepsilon y_x)\|_{\mathcal{Z}_\tau}\\
&\leq \|\Psi^\varepsilon\|\delta_1+\|\Psi^\varepsilon\|\delta_1\sqrt{L}+\|\Psi^\varepsilon\|C_2\|-yy_x+\varepsilon y_x\|_{L^1(0,\tau;L^2(0,L))}\\
&\leq \|\Psi^\varepsilon\|\delta_1+\|\Psi^\varepsilon\|\delta_1\sqrt{L}+\|\Psi^\varepsilon\|C_2\overline{C}\|y\|_{\mathcal{Z}_\tau}^2+\|\Psi^\varepsilon\|C_2\delta_1\\
&=\left(1+\sqrt{L}+C_2\right)\|\Psi^\varepsilon\|\delta_1+C_2\overline{C}\|\Psi^\varepsilon\|r^2.
\end{align*}
Therefore,

\begin{align*}
\|\Gamma y\|_{\mathcal{Z}_\tau}&\leq
C_2\delta_{1}+C_2\left[\left(1+\sqrt{L}+C_2\right)\|\Psi^\varepsilon\|\delta_1+C_2\overline{C}\|\Psi^\varepsilon\|r^2\right]+C_2\left(\overline{C}r^2+\delta_1\right)\\
&=\left[2C_2+C_2\left(1+\sqrt{L}+C_2\right)\|\Psi^\varepsilon\|\right]\delta_1+\left(C_2^2\|\Psi^\varepsilon\|+C_2\right)\overline{C}r^2.
\end{align*}
Choosing
\begin{align*}
r=2\left[2C_2+C_2\left(1+\sqrt{L}+C_2\right)\|\Psi^\varepsilon\|\right]\delta_1
\end{align*}
and $\delta_1$ is small enough such that
\begin{align}\label{choose of delta}
\left(C_2^2\|\Psi^\varepsilon\|+C_2\right)\overline{C}r<\frac{1}{2},\ \ \ 2\left(C_2^2\|\Psi^\varepsilon\|+C_2\right)r<\frac{1}{2},\ \ \ \left(C_2^2\|\Psi^\varepsilon\|+C_2\right)\delta_{1}<\frac{1}{2},
\end{align}
yields that
\begin{align*}
\|\Gamma\|_{\mathcal{Z}_\tau}\leq \frac{r}{2}+\frac{r}{2}=r \implies \Gamma(B)\subset B.
\end{align*}
Additionally, observe that for $y,w \in B$, Proposition \ref{wgproblem} give us
\begin{align*}
\|\Gamma y-\Gamma w\|_{\mathcal{Z}_\tau}=&\|\Lambda_\varepsilon(0,h_y-h_w,0)+\Lambda_\varepsilon(0,0,-yy_x+ww_x+\varepsilon y_x-\varepsilon w_x)\|_{\mathcal{Z}_\tau}\\
\leq& C_2\|h_y-h_w\|_{L^2(0,\tau)}+C_2\|yy_x-ww_x\|_{L^1(0,\tau;L^2(0,L))}\\
&+C_2\|\varepsilon(y_x-w_x)\|_{L^1(0,\tau;L^2(0,L))}.
\end{align*}
Since
\begin{align*}
h_y-h_w&=\Psi^\varepsilon\left(0,-\Lambda_\varepsilon(0,0,-yy_x+\varepsilon y_x)(\cdot,\tau)+\Lambda_\varepsilon(0,0,-ww_x+\varepsilon w_x)(\cdot,\tau)\right)\\
&=\Psi^\varepsilon\left(0,\Lambda_\varepsilon(yy_x-\varepsilon y_x-ww_x+\varepsilon w_x)(\cdot,\tau)\right),
\end{align*}
we have again from Proposition \ref{wgproblem} that
\begin{align*}
C_2\|h_y-h_w\|_{L^2(0,\tau)}\leq& C_2^2\|\Psi^\varepsilon\|\|yy_x-\varepsilon y_x-ww_x+\varepsilon w_x\|_{L^1(0,\tau;L^2(0,L))}\\
\leq& C_2^2\|\Psi^\varepsilon\|\|yy_x-ww_x\|_{L^1(0,\tau;L^2(0,L))}\\
&+C_2^2\|\Psi^\varepsilon\|\|\varepsilon y_x-\varepsilon w_x\|_{L^1(0,\tau;L^2(0,L))}.
\end{align*}
Putting these two previous inequalities together, we find that 
\begin{align*}
	\|\Gamma y-\Gamma w\|_{\mathcal{Z}_\tau}\leq &\left(C_2^2\|\Psi^\varepsilon\|+C_2\right)\|yy_x-ww_x\|_{L^1(0,\tau;L^2(0,L))}\\
	&+\left(C_2^2\|\Psi^\varepsilon\|+C_2\right)\|\varepsilon(y_x-w_x)\|_{L^1(0,\tau;L^2(0,L))}.
\end{align*}
From Lemmas \ref{ywx} and \ref{yyx}, together with the choices \eqref{estimate to epsilon} and \eqref{choose of delta}, it follows that
\begin{align*}
	\|\Gamma y-\Gamma w\|_{\mathcal{Z}_\tau}\leq& 2\left(C_2^2\|\Psi^\varepsilon\|+C_2\right)r\|y-w\|_{\mathcal{Z}_\tau}\\&+\left(C_2^2\|\Psi^\varepsilon\|+C_2\right)C\left(\tau^\frac{1}{2}+\tau^\frac{1}{3}\right)\|\varepsilon\|_{\mathcal{Z}_\tau}\|y-w\|_{\mathcal{Z}_\tau}\\
	\leq &\left[2\left(C_2^2\|\Psi^\varepsilon\|+C_2\right)r+\left(C_2^2\|\Psi^\varepsilon\|+C_2\right)\delta_{1}\right]\|y-w\|_{\mathcal{Z}_\tau}\\
	\leq &\|y-w\|_{\mathcal{Z}_\tau}
\end{align*}
Therefore, $\Gamma:B\rightarrow B$ is a contraction so that, by Banach's fixed point theorem, $\Gamma$ has a fixed point $y\in B$, concluding the proof.
\end{proof}

\vspace{0.2cm}

The second result of this section ensures the construction of solutions for the system \eqref{ncp} on $[2T/3,T]$ starting in one non-null equilibrium and ending near $0$.

\vspace{0.2cm}
\begin{proposition}
	\label{driving-0}
	There exists $\delta_2>0$ such that, for every $d \in (0,\delta_2)$ and  $y_T\in L^2(0,L)$ satisfying $\|y_T\|_{L^2(0,L)}<\delta_{2},$ there exists $h_2 \in L^2(2T/3,T)$ such that, the solution of \eqref{NKdVN} for $t \in [2T/3,T]$ satisfies
	$$
	y(\cdot, 2T/3)=d\text{ \ \ and \ \ }y(\cdot,T/3)=y_T.
	$$
\end{proposition}

\begin{proof}
Let $\delta_{2}\in (0,\epsilon_0)$ be a number to be chosen later. Consider $d \in (0,\delta_{2})$ and $y_T\in L^2(0,L)$ satisfying
$$
	\|y_T\|_{L^2(0,L)}<\delta_2.$$Let $\varepsilon\in (0,\epsilon_0)$ be such that
\begin{align}\label{second estimate to epsilon}
	C\left(\tau^\frac{1}{2}+\tau^\frac{1}{3}\right)\|\varepsilon\|_{\mathcal{Z}_T}<\delta_2\text{ \ \ and \ \ }C\left(\tau^{\frac{1}{2}}+\tau^{\frac{1}{3}}\right)\|\varepsilon\|_{\mathcal{Z}_T}^2<\delta_2,
\end{align}
where $C>0$ is the positive constant given in Lemma \ref{ywx}. If $z$ is a solution to the problem
	\begin{align}\label{non-d}
	\left\{
	\begin{array}{ll}
		z_t+z_x+z_{xxx}+zz_x=0,&\ \text{ in }(0,L)\times(0,\tau),\\
		z_{xx}(0,t)=0,\ z_{x}(L,t)=h(t),\ z_{xx}(L,t)=0&\ \text{ in }(0,\tau),\\
		z(x,0)=d,&\ \text{ in }(0,L),
	\end{array}
	\right.
\end{align}
then $z$ is a solution of
	\begin{align*}
		\left\{
		\begin{array}{ll}
			z_t+(1+\varepsilon)z_x+z_{xxx}=-zz_x+\varepsilon z_x,&\ \text{ in }(0,L)\times(0,\tau),\\
			z_{xx}(0,t)=0,\ z_{x}(L,t)=h(t),\ z_{xx}(L,t)=0,&\ \text{ in }(0,\tau),\\
			z(x,0)=d,&\ \text{ in }(0,L),
		\end{array}
		\right.
\end{align*}
that is,
	\begin{align}\label{ex-d}
		z=\Lambda_\varepsilon(d,h,-zz_x+\varepsilon z_x)=\Lambda_\varepsilon(d,h,0)+\Lambda_\varepsilon(0,0,-zz_x+\varepsilon z_x).
	\end{align}
	
	Given $z\in \mathcal{Z}_\tau$, let $h_z\in L^2(0,\tau)$ defined by
	$$
	h_z=\Psi^\varepsilon\big(d,y_T-\Lambda_\varepsilon(0,0,-zz_x+\varepsilon z_x)(\cdot,\tau)\big).
	$$
Now, consider the map $\Gamma:\mathcal{Z}_\tau\rightarrow\mathcal{Z}_\tau$ given by
	\begin{align*}
		\Gamma z=\Lambda_\varepsilon(d,h_z,0)+\Lambda_\varepsilon(0,0,-zz_x+\varepsilon z_x).
	\end{align*}
Once again, if $\Gamma$ has a fixed point $z$, from the above construction, it follows that $z$ is a solution of \eqref{non-d} with the control $h_z$. Moreover, from \eqref{ex-d} we have
	\begin{align*}
		z=\Lambda_\varepsilon(d,h_z,0)+\Lambda_\varepsilon(0,0,-zz_x+\varepsilon z_x)
	\end{align*}
	so, by definitions of $\Lambda_\varepsilon,\ h_z$ and $\Psi^\varepsilon$, it follows that
	\begin{align*}
	z(\cdot,0)&=\Lambda_\varepsilon(d,h_z,0)(\cdot,0)+\Lambda_\varepsilon(0,0,-zz_x+\varepsilon z_x)(\cdot,0)=d+0=d
	\end{align*}
and
	\begin{align*}
		z(\cdot,\tau)&=\Lambda_\varepsilon(d,h_z,0)(\cdot,\tau)+\Lambda_\varepsilon(0,0,-zz_x+\varepsilon z_x)(\cdot,\tau)\\
		&=y_T-\Lambda_\varepsilon(0,0,-zz_x+\varepsilon z_x)(\cdot,\tau)+\Lambda_\varepsilon(0,0,-zz_x+\varepsilon z_x)(\cdot,\tau)=y_T.
	\end{align*}
	Hence our issue would be solved defining $y:[0,L]\times [2T/3,T]\rightarrow\mathbb{R}$ by
	$$
	y(x,t)=z(x,t-2T/3).
	$$
	
	Now, our focus is to show that $\Gamma$ has a fixed point in a suitable metric space. To do that, consider the set $B$ given by
	\begin{align*}
		B=\left\{z \in \mathcal{Z}_\tau;\ \|z\|_{\mathcal{Z}_\tau}\leq r\right\},
	\end{align*}
	with $r>0$ to be chosen later. By \eqref{ex-d} and Proposition \ref{wgproblem} (together with Remark \ref{mia}) we have, for $z \in B$, that
	\begin{equation*}
		\begin{split}
			\|\Gamma z\|_{\mathcal{Z}_\tau}&\leq \|\Lambda_\varepsilon(d,h_z,0)\|_{\mathcal{Z}_\tau}+\|\Lambda_\varepsilon(0,0,-zz_x+\varepsilon z_x)\|_{\mathcal{Z}_\tau}\\
			&\leq C_2\left(\|d\|_{L^2(0,L)}+\|h_z\|_{L^2(0,\tau)}+\|-zz_x+\varepsilon z_x\|_{L^1(0,\tau;L^2(0,L))}\right).
		\end{split}
	\end{equation*}
As in the proof of the Proposition \ref{d-0to-c},
	\begin{align*}
		\|zz_x-\varepsilon z_x\|_{L^1(0,\tau,L^2(0,L))}&\leq\overline{C}\|z\|_{\mathcal{Z}_\tau}^2+\delta_2.
	\end{align*}
Moreover, 
	\begin{align*}
		\|h_z\|_{L^2(0,\tau)}&=\|\Psi^\varepsilon\big(d,y_T-\Lambda_\varepsilon(0,0,-zz_x+\varepsilon z_x)(\cdot,\tau)\big)\|_{L^2(0,\tau)}\\
		&\leq \|\Psi^\varepsilon\|\left(\|d\|_{L^2(0,L)}+\|y_T\|_{L^2(0,L)}+\|\Lambda_\varepsilon(0,0,-zz_x+\varepsilon z_x)(\cdot,\tau)\|_{L^2(0,L)}\right)\\
		&\leq \|\Psi^\varepsilon\|d\sqrt{L}+\|\Psi^\varepsilon\|\delta_2+\|\Psi^\varepsilon\|\|\Lambda_\varepsilon(0,0,-zz_x+\varepsilon z_x)\|_{\mathcal{Z}_\tau}\\
		&\leq \|\Psi^\varepsilon\|\delta_2+\|\Psi^\varepsilon\|\delta_2\sqrt{L}+\|\Psi^\varepsilon\|C_2\|-zz_x+\varepsilon z_x\|_{L^1(0,\tau;L^2(0,L))}\\
		&\leq \|\Psi^\varepsilon\|\delta_2+\|\Psi^\varepsilon\|\delta_2\sqrt{L}+\|\Psi^\varepsilon\|C_2\overline{C}\|z\|_{\mathcal{Z}_\tau}^2+\|\Psi^\varepsilon\|C_2\delta_2\\
		&=\left(1+\sqrt{L}+C_2\right)\|\Psi^\varepsilon\|\delta_2+C_2\overline{C}\|\Psi^\varepsilon\|r^2.
	\end{align*}
	Therefore,
	
	\begin{align*}
		\|\Gamma z\|_{\mathcal{Z}_\tau}&\leq
		C_2\delta_{2}\sqrt{L}+C_2\left[\left(1+\sqrt{L}+C_2\right)\|\Psi^\varepsilon\|\delta_2+C_2\overline{C}\|\Psi^\varepsilon\|r^2\right]+C_2\left(\overline{C}r^2+\delta_2\right)\\
		&=\left[C_2\sqrt{L}+C_2\left(1+\sqrt{L}+C_2\right)\|\Psi^\varepsilon\|+C_2\right]\delta_2+\left(C_2^2\|\Psi^\varepsilon\|+C_2\right)\overline{C}r^2.
	\end{align*}
	Choosing
	\begin{align*}
		r=2\left[C_2\sqrt{L}+C_2\left(1+\sqrt{L}+C_2\right)\|\Psi^\varepsilon\|+C_2\right]\delta_2
	\end{align*}
	and $\delta_2$ is small enough such that
	\begin{align}\label{choose of delta2}
		\left(C_2^2\|\Psi^\varepsilon\|+C_2\right)\overline{C}r<\frac{1}{2},\ \ \ 2\left(C_2^2\|\Psi^\varepsilon\|+C_2\right)r<\frac{1}{2},\ \ \ \left(C_2^2\|\Psi^\varepsilon\|+C_2\right)\delta_{1}<\frac{1}{2},
	\end{align}
we get that
	\begin{align*}
		\|\Gamma z\|_{\mathcal{Z}_\tau}\leq \frac{r}{2}+\frac{r}{2}=r \implies \Gamma(B)\subset B
	\end{align*}
Furthermore, observe that for $z,w \in B$, Proposition \ref{wgproblem} give us
	\begin{align*}
		\|\Gamma z-\Gamma w\|_{\mathcal{Z}_\tau}=&\|\Lambda_\varepsilon(0,h_z-h_w,0)+\Lambda_\varepsilon(0,0,-zz_x+ww_x+\varepsilon z_x-\varepsilon w_x)\|_{\mathcal{Z}_\tau}\\
		\leq& C_2\|h_z-h_w\|_{L^2(0,\tau)}+C_2\|zz_x-ww_x\|_{L^1(0,\tau;L^2(0,L))}\\
		&+C_2\|\varepsilon(z_x-w_x)\|_{L^1(0,\tau;L^2(0,L))}.
	\end{align*}
	Since
	\begin{align*}
		h_z-h_w&=\Psi^\varepsilon\left(0,-\Lambda_\varepsilon(0,0,-zz_x+\varepsilon z_x)(\cdot,\tau)+\Lambda_\varepsilon(0,0,-ww_x+\varepsilon w_x)(\cdot,\tau)\right)\\
		&=\Psi^\varepsilon\left(0,\Lambda_\varepsilon(zz_x-\varepsilon z_x-ww_x+\varepsilon w_x)(\cdot,\tau)\right),
	\end{align*}
	we have, again from Proposition \ref{wgproblem}, that
	\begin{align*}
		C_2\|h_z-h_w\|_{L^2(0,\tau)}\leq& C_2^2\|\Psi^\varepsilon\|\|zz_x-\varepsilon z_x-ww_x+\varepsilon w_x\|_{L^1(0,\tau;L^2(0,L))}\\
		\leq& C_2^2\|\Psi^\varepsilon\|\|zz_x-ww_x\|_{L^1(0,\tau;L^2(0,L))}\\
		&+C_2^2\|\Psi^\varepsilon\|\|\varepsilon z_x-\varepsilon w_x\|_{L^1(0,\tau;L^2(0,L))}.
	\end{align*}
	Then,
	\begin{align*}
		\|\Gamma z-\Gamma w\|_{\mathcal{Z}_\tau}\leq& \left(C_2^2\|\Psi^\varepsilon\|+C_2\right)\|zz_x-ww_x\|_{L^1(0,\tau;L^2(0,L))}\\
		&+\left(C_2^2\|\Psi^\varepsilon\|+C_2\right)\|\varepsilon(z_x-w_x)\|_{L^1(0,\tau;L^2(0,L))}.
	\end{align*}
	From Lemmas \ref{ywx} and \ref{yyx}, together with \eqref{second estimate to epsilon} and \eqref{choose of delta2}, it follows that
	\begin{align*}
		\|\Gamma z-\Gamma w\|_{\mathcal{Z}_\tau}\leq& 2\left(C_2^2\|\Psi^\varepsilon\|+C_2\right)r\|z-w\|_{\mathcal{Z}_\tau}\\
		&+\left(C_2^2\|\Psi^\varepsilon\|+C_2\right)C\left(\tau^\frac{1}{2}+\tau^\frac{1}{3}\right)\|\varepsilon\|_{\mathcal{Z}_\tau}\|z-w\|_{\mathcal{Z}_\tau}\\
		\leq& \left[2\left(C_2^2\|\Psi^\varepsilon\|+C_2\right)r+\left(C_2^2\|\Psi^\varepsilon\|+C_2\right)\delta_{2}\right]\|z-w\|_{\mathcal{Z}_\tau}\\
		\leq& \|z-w\|_{\mathcal{Z}_\tau}.
	\end{align*}
	Therefore, $\Gamma:B\rightarrow B$ is a contraction so that, by Banach's fixed point theorem, $\Gamma$ has a fixed point $z\in B$, which concludes our proof.
\end{proof}

\begin{remark}
It is important to note that Propositions \ref{d-0to-c} and \ref{driving-0} guarantee that $\|\Psi^\varepsilon\|$ can be chosen uniformly bounded for sufficiently small $\varepsilon$. This uniform bound is crucial to ensure the existence of a fixed point in both theorems.
\end{remark}

\subsection{Controllability on $\mathcal{R}_c$}

We are in a position to prove Theorems \ref{c-length} and \ref{c-length-a}. For the sake of simplicity, we will give the proof of the case $L\in\mathcal{R}_0$ (Theorem \ref{c-length}), and the case $L\in\mathcal{R}_c$ (Theorem \ref{c-length-a}) follows similarly.

\vspace{0.2cm}

\begin{proof}\textbf{(Proof of Theorem \ref{c-length}.)}  Let $\delta_1$ and $\delta_2$ be positive real numbers given in Propositions \ref{d-0to-c} and \ref{driving-0}, respectively. Define $\delta:=\min\{\delta_1,\delta_2\}$ and consider $d\in(0,\delta)$ and $y_0,y_T\in L^2(0,L)$ such that
\begin{align*}
\|y_0\|_{L^2(0,L)},\|y_T\|_{L^2(0,L)}<\delta.
\end{align*}
From Proposition \ref{d-0to-c} there exists $h_1\in L^2(0,T/3)$ such that, the solution $y_1\in \mathcal{Z}_{T/3}$ of
\begin{align*}
	\left\{
	\begin{array}{ll}
		y_t+y_x+y_{xxx}+yy_x=0,&\ \text{ in }(0,L)\times(0,T/3),\\
		y_{xx}(0,t)=0,\ y_{x}(L,t)=h_1(t),\ y_{xx}(L,t)=0,&\ \text{ in }(0,T/3),\\
		y(x,0)=y_0(x),&\ \text{ in }(0,L),
	\end{array}
	\right.
\end{align*}
satisfies
\begin{align*}
y_1(x,T/3)=d.
\end{align*}
On the other hand, thanks to the Proposition \ref{driving-0}, there exists $h_2\in L^2(2T/3,T)$ such that, the solution $y_2\in \mathcal{Z}_{2T/3,T}$ of
\begin{align*}
	\left\{
	\begin{array}{ll}
		y_t+y_x+y_{xxx}+yy_x=0,&\ \text{ in }(0,L)\times(2T/3,T),\\
		y_{xx}(0,t)=0,\ y_{x}(L,t)=h_2(t),\ y_{xx}(L,t)=0,&\ \text{ in }(2T/3,T),\\
		y(x,2T/3)=d,&\ \text{ in }(0,L),
	\end{array}
	\right.
\end{align*}
satisfies
\begin{align*}
y_2(x,T)=y_T.
\end{align*}
Defining $y:[0,L]\times [0,T]\rightarrow\mathbb{R}$ by
\begin{align*}
y=\left\{
\begin{array}{ll}
y_1,&\text{ in }[0,T/3],\\
d,&\text{ in }[T/3,2T/3],\\
y_2,&\text{ in }[2T/3,T],
\end{array}
\right.
\end{align*}
we have that $y\in \mathcal{Z}_T$ and $y$ is solution of \eqref{ncp} driving $y_0$ to $y_T$ at time $T$, showing that the system \eqref{ncp} is exactly controllable, and the proof is completed. 
\end{proof}

\section{Conclusion and open issues}\label{sec4}

In this work, we established the exact boundary controllability of the nonlinear Korteweg--de Vries equation under Neumann boundary conditions in the critical length case. By combining the return method of Coron \cite{CoronBook} with a suitable fixed-point argument, we extended previous results in the literature that were restricted to the subcritical regime. Our analysis shows that, for small perturbations of the steady state $u \equiv c$, the system is exactly controllable in $L^{2}(0,L)$ when $L \in \mathcal{R}_c$. This completes the controllability theory for the Neumann case and opens perspectives for future research on control and several important conclusions, which we will mention now. 

\subsection{Control cost and small-time behavior} 

Observe that Theorems \ref{c-length} and \ref{c-length-a}, and consequently Theorem \ref{c-ap}, guarantee controllability in small time. This is in contrast with the setting considered in \cite{CaCaZh,Rosier}, where the control cost cannot be directly characterized. In our approach, small-time controllability follows from Propositions \ref{d-0to-c} and \ref{driving-0}, specifically from the assumptions on $\tau$ given in \eqref{estimate to epsilon} and \eqref{second estimate to epsilon}.

This naturally leads to the question of characterizing the constant in the observability inequality associated with the control problem \eqref{NKdVN}. Moreover, we expect the control $h$ to blow up in the $L^{2}$--norm as $T \to 0^{+}$, since the control is bounded by the constant
$$\overline{C}:=\frac{3C}{2}\left(T^{\frac{1}{2}}+T^{\frac{1}{3}}\right).$$
Both questions remain open and provide interesting directions for future research on the control problem.

\subsection{Neumann \textit{versus} Dirichlet/mixed controls}

In this work, we investigated the Korteweg–de Vries (KdV) equation on a bounded domain under Neumann boundary conditions. In the classical Hilbert Uniqueness Method (HUM), the main difficulty lies in the fact that the boundary conditions for the adjoint system differ from those of the original system. For this reason, the return method cannot be applied directly, unlike approaches based on series expansions as in the works of Crépeau, Cerpa, Coron, and others \cite{CoCre,cerpa2}. To overcome this issue, we adapted Glass’s idea \cite{Glass} and employed a fixed-point argument applied to a specifically constructed trajectory, which is, in general, not unique.

Several open problems remain concerning the KdV equation with Neumann boundary conditions. For instance, it is still unclear how to use power series expansions to obtain similar results as in \cite{CoCre} or how to define a manifold (a subset of $L^2(0,L)$) associated with the set $\mathcal{R}_c$ to characterize all admissible trajectories in the critical length regime, see for instance the references \cite{cerpa1,cerpa2}. In this sense, the present work not only extends the existing literature but also highlights important directions for future research.
\color{black}

\section*{Acknowledgment}  The authors are grateful to the referee for the careful reading of this paper and their valuable suggestions and comments.  Capistrano-Filho was supported by CAPES grant number 88881.311964/2018-01, COFECUB/CAPES grant number  8887.879175/2023-00, CNPq grants numbers  307808/2021-1 and 401003/2022-1.  Da Silva acknowledges support from CNPq.  This work is part of the Ph.D. thesis of da Silva at the Department of Mathematics of the Universidade Federal de Pernambuco.

\end{document}